\documentclass[12pt,english,british]{amsart}
\usepackage[T1]{fontenc}
\usepackage[latin9]{inputenc}
\usepackage{geometry}
\usepackage{babel}
\usepackage{amstext}
\usepackage{amsthm}
\usepackage{amssymb}
\usepackage{stackrel}
\usepackage{setspace}
\usepackage[unicode=true,pdfusetitle,
 bookmarks=true,bookmarksnumbered=false,bookmarksopen=false,
 breaklinks=false,pdfborder={0 0 1},backref=false,colorlinks=false]
 {hyperref}

\makeatletter

\newcommand{\noun}[1]{\textsc{#1}}

\numberwithin{equation}{section}
\numberwithin{figure}{section}
\theoremstyle{plain}
\newtheorem{thm}{\protect\theoremname}[section]
\theoremstyle{definition}
\newtheorem{defn}[thm]{\protect\definitionname}
\theoremstyle{plain}
\newtheorem{lem}[thm]{\protect\lemmaname}
\theoremstyle{remark}
\newtheorem{rem}[thm]{\protect\remarkname}
\theoremstyle{plain}
\newtheorem{prop}[thm]{\protect\propositionname}

\makeatother

\addto\captionsbritish{\renewcommand{\definitionname}{Definition}}
\addto\captionsbritish{\renewcommand{\lemmaname}{Lemma}}
\addto\captionsbritish{\renewcommand{\propositionname}{Proposition}}
\addto\captionsbritish{\renewcommand{\remarkname}{Remark}}
\addto\captionsbritish{\renewcommand{\theoremname}{Theorem}}
\addto\captionsenglish{\renewcommand{\definitionname}{Definition}}
\addto\captionsenglish{\renewcommand{\lemmaname}{Lemma}}
\addto\captionsenglish{\renewcommand{\propositionname}{Proposition}}
\addto\captionsenglish{\renewcommand{\remarkname}{Remark}}
\addto\captionsenglish{\renewcommand{\theoremname}{Theorem}}
\providecommand{\definitionname}{Definition}
\providecommand{\lemmaname}{Lemma}
\providecommand{\propositionname}{Proposition}
\providecommand{\remarkname}{Remark}
\providecommand{\theoremname}{Theorem}

\begin{document}

\title{Lie structure of associative algebras containing matrix subalgebras}

\author{Alexander Baranov}

\address{Department of Mathematics, University of Leicester, Leicester, LE1
7RH, UK}

\email{ab155@le.ac.uk}
\begin{abstract}
We prove that the commutator subalgebra of the associative algebra
containing a matrix subalgebra is perfect.
\end{abstract}

\maketitle
\global\long\def\bbR{\mathbb{R}}

\global\long\def\ccR{\mathcal{R}}

\begin{singlespace}
\global\long\def\bbF{\mathbb{F}}

\end{singlespace}

\global\long\def\ccF{\mathcal{F}}

\global\long\def\ccL{\mathcal{L}}

\global\long\def\ccP{\mathcal{P}}

\global\long\def\ccB{\mathcal{B}}

\global\long\def\ccS{\mathcal{S}}

\begin{singlespace}
\global\long\def\dlim{\operatorname{\underrightarrow{{\rm lim}}}}

\global\long\def\ker{\operatorname{\rm ker}}

\global\long\def\Im{\operatorname{\rm Im}}

\global\long\def\End{\operatorname{\rm End}}

\global\long\def\dim{\operatorname{\rm dim}}

\global\long\def\core{\operatorname{\rm core}}
 
\end{singlespace}

\global\long\def\skew{\operatorname{\rm skew}}

\global\long\def\Soc{\operatorname{\rm Soc}}

\global\long\def\Range{\operatorname{\rm Range}}

\global\long\def\rank{\operatorname{\rm rank}}

\global\long\def\ad{\operatorname{\rm ad}}

\global\long\def\gl{\operatorname{\rm gl}}

\global\long\def\sl{\operatorname{\rm sl}}

\global\long\def\sp{\operatorname{\rm sp}}

\global\long\def\so{\operatorname{\rm so}}

\global\long\def\u{\operatorname{\rm u^{*}}}

\global\long\def\su{\operatorname{\rm su^{*}}}

\global\long\def\rad{\operatorname{\rm rad}}

\section{Introduction}

The ground field $\bbF$ is algebraically closed of characteristic
$p\ge0$. Throughout the paper, $A$ is a (non-unital) associative
algebra over $\bbF$ containing a non-zero semisimple finite dimensional
subalgebra $S$. Recall that $A$ becomes a Lie algebra $A^{(-)}$
under the Lie bracket $[x,y]=xy-yx$. We denote by $A^{(1)}=[A^{(-)},A^{(-)}]$
its derived subalgebra. 

Recall that a Lie algebra $L$ is said to be \emph{perfect} if $[L,L]=L$.
We say that an associative algebra $A$ is\emph{ $k$-perfect} ($k\ge1)$
if $A$ has no proper ideals of codimension $\le k$. Let $S$ be
a finite dimensional semisimple algebra over $F$. Then it is the
direct sum of matrix ideals, i.e. $S=M_{n_{1}}(\bbF)\oplus\dots\oplus M_{n_{s}}(\bbF)$.
Thus, $S$ is $k$-perfect if and only if $n_{i}>\sqrt{k}$ for all
$i=1,\dots,s$. 

One of our main results is the following theorem. 
\begin{thm}
\label{thm:main} Let $A$ be an associative algebra over $\bbF$
containing a non-zero semisimple finite dimensional subalgebra $S$.
Suppose $A$ is generated by $S$ as an ideal and $S$ is either $1$-perfect
with $p\ne2$ or $4$-perfect with $p=2$. Then 

(1) $A^{(1)}$ is perfect; 

(2) $A=A^{(1)}A^{(1)}+A^{(1)}$;

(3) $A^{(1)}$ is generated by $S^{(1)}$ as an ideal;

(4) $A^{(1)}$ is $\Gamma$-graded where $\Gamma$ consists of the
roots of $S^{(1)}$ and the weights of the natural and conatural $S^{(1)}$-modules. 
\end{thm}

As a corollary, we get the following result for the finite dimensional
algebras, which is a generalization of \cite[Corollary 6.4]{Bav=000026Zal}
(where the case of $4$-perfect algebras in characteristic zero was
proved). 
\begin{thm}
\label{thm:fd} Let $A$ be a finite dimensional algebra over $\bbF.$
Suppose that $A$ is either $4$-perfect or $1$-perfect with $p\neq2$.
Then $A^{(1)}$ is perfect and $A=A^{(1)}A^{(1)}+A^{(1)}$. 
\end{thm}

As an application, we are going to show that the Lie algebras $A^{(1)}$
which appear in Theorem \ref{thm:main} are actually \emph{root-graded}.
The following definition a slight generalization of Berman and Moody's
definition of root graded Lie algebras in \cite{berman1992lie}.
\begin{defn}
\begin{singlespace}
\noindent Let $\Delta$ be a root system and let $\Gamma$ be a finite
set of integral weights of $\Delta$ containing $\Delta$ and $\{0\}$.
A Lie algebra $L$ over a field $\bbF$ of characteristic zero is
said to be \emph{$(\Gamma,\mathfrak{g})$-graded }(or simply\emph{
$\Gamma$-graded}) if 

\noindent $(\Gamma1)$ $L$ contains as a subalgebra a finite-dimensional
split semisimple Lie algebra 
\[
\mathit{\mathrm{\mathfrak{g}=\mathfrak{h}}\oplus\mathrm{\underset{\alpha\in\mathit{\text{\ensuremath{\Delta}}}}{\bigoplus}}}\mathit{\mathrm{\mathfrak{g}}}_{\alpha},
\]
 whose root system is $\Delta$ relative to a split Cartan subalgebra
$\text{\ensuremath{\mathfrak{h}}}=\mathfrak{g}_{0}$;

\noindent $(\Gamma2)$ $L=\underset{\alpha\in\Gamma}{\bigoplus}L_{\alpha}$
where $L_{\alpha}=\left\{ x\in\mathrm{\mathit{L}\mid\left[\mathit{h,x}\right]=\mathit{\alpha\left(h\right)x\text{ for all }h\in\mathrm{\mathit{\mathfrak{h}}}}}\right\} $;

\noindent $(\Gamma3)$ $L_{0}=\underset{\alpha,-\alpha\in\Gamma\setminus\{0\}}{\overset{}{\sum}}\left[L_{\alpha},L_{-\alpha}\right]$.
\end{singlespace}
\end{defn}

\begin{thm}
\label{thm:rg} Let $A$ and $S$ be as in Theorem \ref{thm:main}.
Let $Q_{1},\dots,Q_{t}$ be the simple components of the Lie algebra
$S^{(1)}$. Then the Lie algebra $A^{(1)}$ is $\Gamma$-graded where
$\Gamma$ consists of the roots of $S^{(1)}$ and all weights of the
form $\lambda_{i}+\lambda_{j}$ where $1\le i<j\le t$ and $\lambda_{i}$
is either zero or one of the weights of the natural or conatural $Q_{i}$-modules. 
\end{thm}

\section{$S$-decomposition of $A$}

Throughout the paper, $A$ is a (non-unital) associative algebra over
$\bbF$ containing a non-zero semisimple finite dimensional subalgebra
$S$. We denote by $\hat{A}$ the unital algebra obtained by adjoining
an identity element $\mathbf{1}=\mathbf{1}_{\hat{A}}$ to $A$, i.e.
$\hat{A}=A\oplus\bbF\mathbf{1}$ as a vector space and $A$ is an
ideal of $\hat{A}$ of codimension 1. Similarly, we denote by $\hat{S}$
the unital semisimple subalgebra $S\oplus\bbF\mathbf{1}$ of $\hat{A}$.
We denote by $M_{n}$ the algebra of $n\times n$ matrices over $\bbF$. 

Let $\{S_{i}:i\in I\}$ be the set of the simple components of $S$.
Then the set of the simple components of $\hat{S}$ is $\{S_{i}:i\in\hat{I}\}$
where $\hat{I}=I\cup\{0\}$ and $S_{0}\cong\bbF$. Let $\mathbf{1}_{S_{i}}$
be the identity element of $S_{i}$, $i\in\hat{I}$. Then $S_{0}=\bbF\mathbf{1}_{S_{0}}$,
$\mathbf{1}=\sum_{i\in\hat{I}}\mathbf{1}_{S_{i}}$ and $\{\mathbf{1}_{S_{i}}\mid i\in\hat{I}\}$
is the complete set of the primitive central idempotents of $\hat{S}$.
We identify each $S_{i}$ with the matrix algebra $M_{n_{i}}$ (so
$\dim S_{i}=n_{i}^{2}$). Let $V_{i}$ be the natural left $S_{i}$-module.
Then the dual space $V_{i}^{*}$ is the natural right $S_{i}$-module.
We identify the space $V_{i}$ (resp. $V_{i}^{*}$) with the column
(resp. row) space $\bbF^{n_{i}}$. 

Let $M$ be a (not necessarily unital) $S$-bimodule. Set $\mathbf{1}m=m\mathbf{1}=m$
for all $m\in M$. Then $M$ is a unital $\hat{S}$-bimodule or equivalently,
unital left $\mathcal{S}$-module, where 
\[
\mathcal{S}=\hat{S}\otimes\hat{S}^{op}=\underset{i,j\in\hat{I}}{\bigoplus}S_{i}\otimes S_{j}^{op}
\]
is the\emph{ enveloping algebra }of $S$. Put $\mathcal{S}_{ij}=S_{i}\otimes S_{j}^{op}\cong M_{n_{i}}(\bbF)\otimes M_{n_{j}}(\bbF)\cong M_{n_{i}n_{j}}(\bbF)$
for all $i,j\in\hat{I}$. Then each $\mathcal{S}_{ij}$ is a simple
component of $\mathcal{S}$ and $\mathcal{S}=\bigoplus_{i,j\in\hat{I}}\mathcal{S}_{ij}$
is semisimple. Thus, $M$ is the direct sum (possibly infinite) of
simple left $\mathcal{S}$-modules, or equivalently, $\hat{S}$-bimodules.
Put $V_{ij}=V_{i}\otimes V_{j}^{*}$ for all $i,j\in\hat{I}$. We
identify each $V_{ij}$ with the space of all $n_{i}\times n_{j}$
matrices over $\bbF$. Then $V_{ij}$ is the natural left $\mathcal{S}_{ij}$-module
(resp. $S_{i}$-$S_{j}$-bimodule). Hence, as a left $\mathcal{S}$-module,
$M$ is the direct sum of copies of $V_{ij}$, $i,j\in\hat{I}$. By
collecting together the isomorphic copies, we obtain the following.
\begin{lem}
\label{lem:M bimodule} Let $M$ be an $S$-bimodule. Then, as an
$\hat{S}$-bimodule, 
\[
M\cong\bigoplus_{i,j\in\hat{I}}V_{ij}\otimes\Lambda_{M}(i,j)
\]
for some vector spaces $\Lambda_{M}(i,j)$.
\end{lem}

\begin{rem}
Recall that $S_{0}\cong\bbF$ is 1-dimensional and $\hat{S}=S_{0}\oplus S$.
Therefore, $V_{00}$ is a 1-dimensional $S_{0}$-bimodule, $M_{0}:=V_{00}\otimes\Lambda_{M}(0,0)$
is a\emph{ trivial} $S$-sub-bimodule of $M$ (i.e. $SM_{0}=M_{0}S=0$)
and $M=(SM+MS)\oplus M_{0}$.
\end{rem}

Recall that $S$ is a subalgebra of $A$, so $A$ is an $S$-bimodule.
Hence, by Lemma \ref{lem:M bimodule}, there is an $\hat{S}$-bimodule
isomorphism 
\begin{equation}
\theta:A\rightarrow\bigoplus_{i,j\in\hat{I}}V_{ij}\otimes\Lambda_{A}(i,j),\label{eq:A to Wij.^(i,j)}
\end{equation}
where $\hat{I}=I\cup\{0\}$, $\Lambda_{A}(i,j)$ are vector spaces
and $V_{ij}$ is the space of all $n_{i}\times n_{j}$ matrices over
$\bbF$. Hence, the space 
\[
\mathbf{V}=\underset{i,j\in\hat{I}}{\bigoplus}V_{ij}
\]
is an associative algebra with respect to the standard matrix multiplication
(with $V_{ij}V_{st}=0$ unless $j=s$). We denote by $\{e_{st}^{ij}\mid1\leq s\leq n_{i},\ 1\le t\leq n_{j}\}$
(or simply $e_{st}$) the standard basis of $V_{ij}$ consisting of
matrix units. For example, $\{e_{11},e_{12},\ldots,e_{1n_{j}}\}$
is the basis of $V_{0n_{j}}=V_{0}\otimes V_{j}^{*}$. We will identify
$A$ with $\theta(A)$. 

\selectlanguage{english}%

\selectlanguage{british}%
Consider any two simple $\ccS$-submodules (or equivalently, $\hat{S}$-sub-bimodules)
$W_{ij}=V_{ij}\otimes\lambda_{ij}$ and $W_{st}=V_{st}\otimes\lambda_{st}$
of $A$. Then the product $W_{ij}W_{st}=W_{ij}S_{j}S_{s}W_{st}=0$
unless $j=s$. If $j=s$, then 
\[
W_{ij}W_{jt}=W_{ij}S_{j}W_{jt}=(V_{ij}\otimes\lambda_{ij})S_{j}(V_{jt}\otimes\lambda_{jt})
\]
 which is a homomorphic image of the $\mathcal{S}_{it}$-module 
\[
V_{ij}\otimes_{S_{j}}V_{jt}\cong V_{i}\otimes V_{j}^{*}\otimes_{S_{j}}V_{j}\otimes V_{t}^{*}\cong V_{i}\otimes V_{t}^{*}=V_{it}.
\]
Therefore, 

\begin{equation}
(V_{ij}\otimes\lambda_{ij})(V_{jt}\otimes\lambda_{jt})=V_{it}\otimes\lambda_{it}\label{eq:vtimesv}
\end{equation}
for some $\lambda_{it}\in\Lambda_{A}(i,t)$ (defined up to a scalar
multiple). Note that $e_{11}$ exists in every $V_{ij}$ for $i,j\in\hat{I}$
and the product $(e_{11}\otimes\lambda_{ij})(e_{11}\otimes\lambda_{jt})$
must be equal to a scalar multiple of $e_{11}\otimes\lambda_{it}$.
Thus, by rescaling $\lambda_{it}$ if necessary, we can assume that
there is a binary product $(\lambda_{ij},\lambda_{jt})\mapsto\lambda_{ij}\lambda_{jt}$
such that 
\[
(e_{11}\otimes\lambda_{ij})(e_{11}\otimes\lambda_{jt})=e_{11}\otimes\lambda_{ij}\lambda_{jt}
\]
for all $\lambda_{ij}\in\Lambda_{A}(i,j)$ and $\lambda_{jt}\in\Lambda_{A}(j,t)$.
It is easy to see that this binary product is bilinear and associative
(since the product of $V$'s in (\ref{eq:vtimesv}) is associative).
We get the following.
\begin{lem}
\label{lem:^(i,j)} There is a multiplication structure on the space
$\Lambda_{A}:=\bigoplus_{i,j\in\hat{I}}\Lambda_{A}(i,j)$ satisfying
the following conditions:

(i) $\Lambda_{A}(i,j)\Lambda_{A}(s,t)=0$, if $j\neq s$;

(ii) $\Lambda_{A}(i,j)\Lambda_{A}(j,t)\subseteq\Lambda_{A}(i,t)$;

(iii) $\Lambda_{A}$ is an associative algebra with respect to this
multiplication;

(iv) $S_{i}=V_{ii}\otimes\mathbf{1}_{i}$ ($i\in I$) where $\mathbf{1}_{i}$
is the identity element of the subalgebra $\Lambda_{A}(i,i)$;

(v) $S=\bigoplus_{i\in I}S_{i}=\bigoplus_{i\in I}V_{ii}\otimes\mathbf{1}_{i}$. 
\end{lem}

Since both $\mathbf{V}$ and $\Lambda_{A}$ are associative algebras,
their tensor product $\mathbf{V}\otimes\Lambda_{A}$ is also an associative
algebra. It is easy to see that $\bigoplus_{i,j\in\hat{I}}V_{ij}\otimes\Lambda_{A}(i,j)$
is a subalgebra of $\mathbf{V}\otimes\Lambda_{A}$ with the following
product: for all $X\otimes\lambda\in V_{ij}\otimes\Lambda_{A}(i,j)$
and $Y\otimes\mu\in V_{st}\otimes\Lambda_{A}(s,t)$, 

\[
(X\otimes\lambda)(Y\otimes\mu)=XY\otimes\lambda\mu
\]
where $XY$ is the matrix multiplication (with $XY=0$ if $j\neq s$)
and $\lambda\mu$ is the multiplication in $\Lambda_{A}$. Moreover,
this subalgebra is isomorphic to $A$. Thus we get the following.
\begin{prop}
\label{prop:Vij.^(i,j )  is an algebra} $\theta$ is an isomorphism
of associative algebras.
\end{prop}

\begin{rem}
Suppose that $A$ is finite dimensional and $S$ in Proposition \ref{prop:Vij.^(i,j )  is an algebra}
is a Levi (i.e. maximal semisimple) subalgebra of $A$. Then $\Lambda_{A}/\rad\Lambda_{A}\cong\bigoplus_{i\in I}\bbF\mathbf{1}_{i}$.
In particular, $\Lambda_{A}$ is a basic algebra. Moreover, it is
not difficult to see that $\Lambda_{A}$ is Morita equivalent to $A$.
\end{rem}

We will identify the algebra $A$ with its image $\theta(A)=\bigoplus_{i,j\in\hat{I}}V_{ij}\otimes\Lambda_{A}(i,j)$. 
\begin{defn}
\label{def:S-decomposition} We say that $A=\bigoplus_{i,j\in\hat{I}}V_{ij}\otimes\Lambda_{A}(i,j)$
is the \emph{$S$-decomposition} of $A$.
\end{defn}

Let $A=\bigoplus_{i,j\in\hat{I}}V_{ij}\otimes\Lambda_{A}(i,j)$ be
the $S$-decomposition of $A$ and let $B$ be a subalgebra of $A$
such that $SB+BS\subseteq B$. Then $B$ is an $S$-sub-bimodule of
$A$ and 
\begin{equation}
B=\bigoplus_{i,j\in\hat{I}}V_{ij}\otimes\Lambda_{B}(i,j)\label{eq:Bdec}
\end{equation}
where $\Lambda_{B}=\bigoplus_{i,j\in\hat{I}}\Lambda_{B}(i,j)$ is
a subspace of $\Lambda_{A}=\bigoplus_{i,j\in\hat{I}}(i,j)$. We say
that (\ref{eq:Bdec}) is the \emph{$S$-decomposition} of $B$. Using
Proposition \ref{prop:Vij.^(i,j )  is an algebra}, we obtain the
following. 
\begin{prop}
\label{prop:^B sub ^A} Let $B$ be a subalgebra of $A$ such that
$SB+BS\subseteq B$. Then

(i) $\Lambda_{B}$ is a subalgebra of $\Lambda_{A}$.

(ii) If $B$ is an ideal of $A$, then $\Lambda_{B}$ is an ideal
of $\Lambda_{A}$. 
\end{prop}

\begin{defn}
\label{def:B is S-mod} We say that a subalgebra $B$ of $A$ is $S$\emph{-modgenerated}
if $SB+BS\subseteq B$ and $B$ is generated (as an algebra) by non-trivial
simple $S$-sub-bimodules of $B$.
\end{defn}

Let $B$ be a subalgebra of $A$ such that $SB+BS\subseteq B$ and
let $B_{S}$ be the subalgebra of $B$ generated by non-trivial simple
$S$-sub-bimodules of $B$. Then $B_{S}$ is a subalgebra of $A$
with $SB_{S}+B_{S}S\subseteq B_{S}$. Let $B_{S}=\bigoplus_{i,j\in\hat{I}}V_{ij}\otimes\Lambda_{B_{S}}(i,j)$
be the $S$-decomposition of $B_{S}$. Since $B_{S}$ is a subalgebra
of $B$, we have $\Lambda_{B_{S}}(i,j)\subseteq\Lambda_{B}(i,j)$
for all $i.j\in\hat{I}$. On the other hand, $B_{S}$ is generated
by all $V_{ij}\otimes\Lambda_{B}(i,j)$ with $(i,j)\neq(0,0)$, so
$\Lambda_{B_{S}}(i,j)=\Lambda_{B}(i,j)$ for all $(i,j)\ne(0,0)$.
We have the following proposition. 
\begin{prop}
\label{B_S} (i) $\Lambda_{B_{S}}(0,0)=\sum_{i\in I}\Lambda_{B_{S}}(0,i)\Lambda_{B_{S}}(i,0)$. 

(ii) $\Lambda_{B_{S}}(0,0)=\sum_{i\in I}\Lambda_{B}(0,i)\Lambda_{B}(i,0)$. 

(iii) $B_{S}$ is an ideal of $B$. 
\end{prop}

\begin{proof}
(i) Let $B'_{S}=\bigoplus_{(i,j)\neq(0,0)}V_{ij}\otimes\Lambda_{B_{S}}(i,j)$.
Then $B_{S}$ is generated by $B'_{S}$. Choose any product $\Pi=V_{i_{1}j_{1}}\otimes\Lambda_{B_{S}}(i_{1},j_{1})\ldots V_{i_{k}j_{k}}\otimes\Lambda_{B_{S}}(i_{k},j_{k})$,
where $k\geq2$, in $B_{S}$. Suppose that $\Pi$ is not in $B_{S}$.
Then by Proposition \ref{prop:Vij.^(i,j )  is an algebra}, $j_{1}=i_{2}$,
$j_{2}=i_{3}$, $\ldots$, $j_{k-1}=i_{k}$ and $i_{1}=j_{k}=0$,
so $j_{1}\neq0$. Therefore, by using Proposition \ref{prop:Vij.^(i,j )  is an algebra},
we obtain 
\begin{eqnarray*}
\Pi & = & (V_{0j_{1}}\otimes\Lambda_{B_{S}}(0,j_{1}))((V_{j_{1}j_{2}}\otimes\Lambda_{B_{S}}(j_{1},j_{2})\ldots(V_{j_{k-1}0}\otimes\Lambda_{B_{S}}(j_{k-1},0))\\
 & \subseteq & (V_{0j_{1}}\otimes\Lambda_{B_{S}}(0,j_{1})(V_{j_{1}0}\otimes\Lambda_{B_{S}}(j_{1},0)\subseteq V_{00}\otimes\Lambda_{B_{S}}(0,j_{1})\Lambda_{B_{S}}(j_{1},0).
\end{eqnarray*}
Hence, $\Lambda_{B_{S}}(0,0)=\sum_{i\in I}\Lambda_{B_{S}}(0,i)\Lambda_{B_{S}}(i,0)$,
as required. 

(ii) Follows from (i) since $\Lambda_{B_{S}}(i,j)=\Lambda_{B}(i,j)$
for all $(i,j)\ne(0,0)$. 

(iii) We need to show that $\Lambda_{B}\Lambda_{B_{S}}\subseteq\Lambda_{B_{S}}$
(the case $\Lambda_{B_{S}}\Lambda_{B}\subseteq\Lambda_{B_{S}}$ is
similar). Since $\Lambda_{B_{S}}(i,j)=\Lambda_{B}(i,j)$ for all $(i,j)\ne(0,0)$
and $\Lambda_{B_{S}}$ is an algebra, we have $\Lambda_{B}(i,j)\Lambda_{B_{S}}\subseteq\Lambda_{B_{S}}$
for all $(i,j)\neq(0,0)$. It remain to note that $\Lambda_{B}(0,0)\Lambda_{B_{S}}\subseteq\Lambda_{B_{S}}$.
Indeed, for all $i\in I$ we have 
\[
\Lambda_{B}(0,0)\Lambda_{B_{S}}(0,i)\subseteq\Lambda_{B}(0,0)\Lambda_{B}(0,i)\subseteq\Lambda_{B}(0,i)=\Lambda_{B_{S}}(0,i).
\]
Therefore, $\Lambda_{B_{S}}$ is an ideal of $\Lambda_{B}$. 
\end{proof}
\begin{prop}
\label{prop:(i) =00003D (ii) =00003D (iii)} The following are equivalent

(i) $A$ is $S$-modgenerated. 

(ii) $\Lambda_{A}(0,0)=\sum_{i\in I}\Lambda_{A}(0,i)\Lambda_{A}(i,0)$. 

(iii) The ideal of $A$ generated by $S$ coincides with $A$. 
\end{prop}

\begin{proof}
Let $A_{S}$ be the subalgebra of $A$ generated by non-trivial irreducible
$S$-sub-bimodules of $A$. 

$(i)\Rightarrow(ii)$: Suppose that $A$ is $S$-modgenerated. Then
$A=A_{S}$, so $(ii)$ follows from Proposition \ref{B_S}(ii). 

$(ii)\Rightarrow(i)$: Suppose $(ii)$ holds. Then by Proposition
\ref{B_S}(iii), 
\[
\Lambda_{A_{S}}(0,0)=\sum_{i\in I}\Lambda_{A}(0,i)\Lambda_{A}(i,0)=\Lambda_{A}(0,0).
\]
Therefore, $A_{S}=A$, as required. 

$(ii)\Leftrightarrow(iii)$: Note that the ideal $T$ of $A$ generated
by $S$ contains $A_{S}$. On the other hand, by Proposition \ref{B_S}(i),
$A_{S}$ is an ideal of $A$ containing $S,$ so $T\subseteq A_{S}$.
Therefore, $T=A_{S}$, as required. 
\end{proof}
\begin{defn}
We say that $A$ is \emph{$M_{k}$-modgenerated} if $A$ is $S$-modgenerated
with $S$ isomorphic to the matrix algebra $M_{k}(\bbF)$. 
\end{defn}

\begin{prop}
Suppose that $A$ is $S$-modgenerated and $S$ is $1$-perfect (resp.
$4$-perfect). Then $A$ is $M_{2}$-modgenerated (resp. $M_{3}$-modgenerated). 
\end{prop}

\begin{proof}
Suppose that $S$ is $1$-perfect (the case of $4$-perfect $S$ is
similar). Let $\{S_{i}\mid i\in I\}$ be the set of the simple components
of $S$. Since $S$ is $1$-perfect, each $S_{i}$ contains a subalgebra
$P_{i}\cong M_{2}(\bbF)$. Fix any subalgebra $P$ of $S$ such that
$P\cong M_{2}(\bbF)$ and the projection of $P$ onto $S_{i}$ is
$P_{i}$. Then $S$ is generated by $P$ as an ideal. Since $A$ is
$S$-modgenerated, by Proposition \ref{prop:(i) =00003D (ii) =00003D (iii)},
$A$ is generated by $S$ as an ideal. Therefore, $A$ is generated
by $P$ as an ideal, so $A$ is $M_{2}$-modgenerated. 
\end{proof}
\begin{prop}
\label{levi} Let $A$ be a perfect finite dimensional associative
algebra and let $S$ be a Levi subalgebra of $A$. Then $A$ is $S$-modgenerated. 
\end{prop}

\begin{proof}
Let $T$ be an ideal of $A$ generated by $S$. By Proposition \ref{prop:(i) =00003D (ii) =00003D (iii)},
we need to show that $A=T$. Since $S$ is Levi subalgebra of $A$
and $S\subseteq T$, $A/T$ is nilpotent. As $A$ is perfect, $A/T=0$,
so $T=A$, as required.
\end{proof}

\section{Lie Structure of $A$}

Throughout this section, $A$ is an associative algebra, $S$ is a
finite dimensional semisimple subalgebra of $A$ such that $A$ is
$S$-modgenerated, $A=\bigoplus_{i,j\in\hat{I}}V_{ij}\otimes\Lambda_{A}(i,j)$
is the $S$-decomposition for $A$, $\{S_{i}\mid i\in I\}$ is the
set of the simple components of $S$, each $S_{i}$ is identified
with $M_{n_{i}}(\bbF)$, $B$ is a subalgebra of $A$ such that $BS+SB\subseteq B$,
$B_{S}$ is the subalgebra of $B$ generated by all non-trivial irreducible
$S$-sub-bimodules of $B$, $A=\bigoplus_{i,j\in\hat{I}}V_{ij}\otimes\Lambda_{A}(i,j)$
(resp. $B_{S}=\bigoplus_{i,j\in\hat{I}}V_{ij}\otimes\Lambda_{B_{S}}(i,j)$)
is the $S$-decomposition of $A$ (resp. $B_{S}$), $\{e_{st}^{ij}\mid1\leq s\leq n_{i},\ 1\le t\leq n_{j}\}$
(or simply $e_{st}$) is the standard basis of $V_{ij}$ consisting
of the matrix units and $S=\bigoplus_{i\in I}V_{ii}\otimes\mathbf{1}_{i}$.
In this section we suppose that $S$ is $1$-perfect if $p\ne2$ or
$4$-perfect if $p=2$. In particular, $n_{i}\ge2$ for all $i\in I$
and $n_{0}=1$. 

We define $V'_{i,j}$ ($i,j\in\hat{I}$) as follows: 
\begin{equation}
V'_{ij}=\begin{cases}
\{X\in V_{ii}\mid tr(X)=0\}\qquad, & \text{ if }i=j\\
V_{ij}, & \text{ if }i\neq j
\end{cases}.\label{eq:W'ij}
\end{equation}

Throughout this section, $L$ is the Lie algebra of $A$ generated
by all $V'_{ij}\otimes\Lambda_{A}(i,j)$ and $Q:=[S,S]=\bigoplus_{i\in I}Q_{i}\subseteq L$
where $Q_{i}=V'_{ii}\otimes\mathbf{1}_{i}$ are ideals of the Lie
algebra $Q$. Note that each $Q_{i}\cong sl_{n_{i}}(\bbF$) is simple
if $p=0$ or $p\not|n_{i}$. In the case $p|n_{i}$, the algebra $Q_{i}$
is \emph{quasisimple}, i.e. $Q_{i}$ is perfect and $Q_{i}/Z(Q_{i})$
is simple (here we use that $n_{i}\ge3$ if $p=2$). 
\begin{prop}
\label{prop:L is perfect} $L$ is perfect. 
\end{prop}

\begin{proof}
We need to show that $L\subseteq[L,L]$. It is enough to prove that
$V'_{ij}\otimes\Lambda_{A}(i,j)\subseteq[L,L]$ for all $(i,j)\neq(0,0)$.
Let $V'_{ij}\otimes\lambda\in V'_{ij}\otimes\Lambda_{A}(i,j)$. Suppose
that $i=j$. Since $Q\subseteq L$ and $V'_{ii}\cong Q_{i}$ as a
$Q_{i}$-module, we have 
\begin{eqnarray*}
[L,L] & \supseteq & [V'_{ii}\otimes\mathbf{1}_{i},V'_{ii}\otimes\lambda]=[Q_{i},V'_{ii}\otimes\lambda]=V'_{ii}\otimes\lambda.
\end{eqnarray*}
Suppose now that $i\neq j$. Since $(i,j)\neq(0,0)$, at least one
of the indices, say $i$, is non-zero. Since $V'_{ij}$ is isomorphic
to the direct sum of $n_{j}$ copies of $V_{i}$ as $Q_{i}$-module,
we have 
\begin{eqnarray*}
[L,L] & \supseteq & [V'_{ii}\otimes\mathbf{1}_{i},V'_{ij}\otimes\lambda]=[Q_{i},V'_{ij}\otimes\lambda]=V'_{ij}\otimes\lambda,
\end{eqnarray*}
as required. Therefore, $L$ is perfect.
\end{proof}

\begin{lem}
\label{lem:Vii * ^(i,i) } $[V_{ii}\otimes\Lambda_{A}(i,i),V_{ii}\otimes\Lambda_{A}(i,i)]\subseteq L$
for all $i\in I$. 
\end{lem}

\begin{proof}
Denote by \textbf{$\mathbf{\varepsilon}_{i}$} the identity matrix
in $V_{ii}$. Let\noun{ }$X\otimes\lambda,Y\otimes\mu\in V_{ii}\otimes\Lambda_{A}(i,i)$.
We wish to show that $[X\otimes\lambda,Y\otimes\mu]\in L$. Recall
that $[L,L]=L$. We have two cases depending on whether $p$ divides
$n_{i}$ or not. 

Suppose first that $p$ does not divide $n_{i}$. In this case $\varepsilon_{i}\not\in L$
and $V_{ii}=V'_{ii}\oplus\bbF\varepsilon_{i}$. Let $X',Y'\in V'_{ii}$.
By linearity, it suffices to show that $[X'\otimes\lambda,Y'\otimes\mu]$,
$[X'\otimes\lambda,\varepsilon_{i}\otimes\mu]$ and $[\varepsilon_{i}\otimes\lambda,\varepsilon_{i}\otimes\mu]$
are in $L$. We have $[X'\otimes\lambda,Y'\otimes\mu]\in[L,L]=L$
and $[X'\otimes\lambda,\varepsilon_{i}\otimes\mu]=X'\otimes\lambda\mu-X'\otimes\mu\lambda\in L$,
as required. Note that $e_{rq}\in V'_{ii}$ for all $r\neq q$, so
$[e_{rq}\otimes\lambda,e_{qr}\otimes\mu]\in[L,L]=L$. Hence, we have
the following system of commutators in $L$:
\[
\begin{array}{ccc}
[e_{12}\otimes\lambda,e_{21}\otimes\mu] & = & e_{11}\otimes\lambda\mu-e_{22}\otimes\mu\lambda\in L,\\
{}[e_{23}\otimes\lambda,e_{32}\otimes\mu] & = & e_{22}\otimes\lambda\mu-e_{33}\otimes\mu\lambda\in L,\\
\vdots & \vdots & \vdots\\
{}[e_{n-1,n}\otimes\lambda,e_{n,n-1}\otimes\mu] & = & e_{n-1,n-1}\otimes\lambda\mu-e_{nn}\otimes\mu\lambda\in L,\\
{}[e_{n1}\otimes\lambda,e_{1n}\otimes\mu] & = & e_{nn}\otimes\lambda\mu-e_{11}\otimes\mu\lambda\in L
\end{array}
\]
where $n=n_{i}$. Since the sum of all expressions on the right equals
to $\varepsilon_{i}\otimes\lambda\mu-\varepsilon_{i}\otimes\mu\lambda=[\varepsilon_{i}\otimes\lambda,\varepsilon_{i}\otimes\mu]$
we have $[\varepsilon_{i}\otimes\lambda,\varepsilon_{i}\otimes\mu]\in L$,
as required. 

Now, suppose that $p$ divides $n_{i}$. In this case we have $\varepsilon_{i}\in V'_{ii}$
and $V_{ii}=V'_{ii}+\bbF e_{11}$. Let $X',Y'\in V'_{ii}$. By linearity,
it suffices to show that $[X'\otimes\lambda,Y'\otimes\mu]$, $[X'\otimes\lambda,e_{11}\otimes\mu]$
and $[e_{11}\otimes\lambda,e_{11}\otimes\mu]$ are in $L$. As before,
$[X'\otimes\lambda,Y'\otimes\mu]\in[L,L]=L$ and $[X'\otimes\lambda,e_{11}\otimes\mu]=X'e_{11}\otimes\lambda\mu-e_{11}X'\otimes\mu\lambda$.
Put $X'e_{11}=X_{1}+\beta e_{11}$ and $e_{11}X'=X_{2}+\beta e_{11}$,
where $X_{1},X_{2}\in V'_{ii}$. Then 
\[
[X'\otimes\lambda,e_{11}\otimes\mu]=X_{1}\otimes\lambda\mu+\beta e_{11}\otimes\lambda\mu-X_{2}\otimes\mu\lambda-\beta e_{11}\otimes\mu\lambda.
\]
Note that $X_{1}\otimes\lambda\mu,X_{2}\otimes\mu\lambda\in L$. Hence,
it remains to show that $e_{11}\otimes\lambda\mu-e_{11}\otimes\mu\lambda\in L$.
Recall that $e_{rq}\in V'_{ii}$ for all $r\neq q$. Hence, 
\[
[e_{rq}\otimes\lambda,e_{qr}\otimes\mu]\in[V'_{ii}\otimes\Lambda_{A}(i,i),V'\otimes\Lambda_{A}(i,i)]\subseteq[L,L]=L.
\]
Hence, we have the following system of commutators 

\[
\begin{array}{ccc}
[e_{12}\otimes\mu,e_{21}\otimes\lambda] & = & e_{11}\otimes\mu\lambda-e_{22}\otimes\lambda\mu\in L,\\
{}[e_{13}\otimes\mu,e_{31}\otimes\lambda] & = & e_{11}\otimes\mu\lambda-e_{33}\otimes\lambda\mu\in L,\\
\vdots & \vdots & \vdots\\
{}[e_{1,n}\otimes\mu,e_{n,1}\otimes\lambda] & = & e_{11}\otimes\mu\lambda-e_{nn}\otimes\lambda\mu\in L,
\end{array}
\]
where $n=n_{i}$. The sum of all these commutators $\stackrel[k=2]{n}{\sum}[e_{1k}\otimes\mu,e_{k1}\otimes\lambda]\in L$.
Since $p|n_{i}$, we have $ne_{11}=0$, so this sum is 
\[
\stackrel[k=2]{n}{\sum}(e_{11}\otimes\mu\lambda-e_{kk}\otimes\lambda\mu)=-e_{11}\otimes\mu\lambda+e_{11}\otimes\lambda\mu-\varepsilon_{i}\otimes\lambda\mu\in L.
\]
Hence, $e_{11}\otimes\lambda\mu-e_{11}\otimes\mu\lambda\in L$, as
required. 
\end{proof}
\begin{lem}
\label{lem:=00005BL,Vii * ^A(i,i)=00005D sub L} $[L,V_{ii}\otimes\Lambda_{A}(i,i)]\subseteq L$
for all $i\in I$. 
\end{lem}

\begin{proof}
Note that 
\[
[L,V_{ii}\otimes\Lambda_{A}(i,i)]\subseteq[A,V_{ii}\otimes\Lambda_{A}(i,i)]=\bigoplus_{(i,j)\neq(0,0)}[V_{ij}\otimes\Lambda_{A}(i,j),V_{ii}\otimes\Lambda_{A}(i,i)]].
\]
Hence, we have two cases depending on the values of $i$ and $j$.
If $i=j$, then by Lemma \ref{lem:Vii * ^(i,i) }, $[V_{ii}\otimes\Lambda_{A}(i,i),V_{ii}\otimes\Lambda_{A}(i,i)]\subseteq L$.
Suppose that $i\neq j$. Then by using (\ref{eq:W'ij}) we obtain
\[
[V_{ij}\otimes\Lambda_{A}(i,j),V_{ii}\otimes\Lambda_{A}(i,i)]]\subseteq V_{ij}\otimes\Lambda_{A}(i,j)=V'_{ij}\otimes\Lambda_{A}(i,j)\subseteq L.
\]
Thus, $[L,V_{ii}\otimes\Lambda_{A}(i,i)]\subseteq L$.
\end{proof}
\begin{lem}
\label{lem:V00 * ^(0,0) } $[V_{00}\otimes\Lambda_{A}(0,0),V_{00}\otimes\Lambda_{A}(0,0)]\subseteq L$. 
\end{lem}

\begin{proof}
Let $X\otimes\lambda,Y\otimes\mu\in V_{00}\otimes\Lambda_{A}(0,0)$.
By Proposition \ref{prop:(i) =00003D (ii) =00003D (iii)}, we have
\begin{equation}
\Lambda_{A}(0,0)=\sum_{i\in I}\Lambda_{A}(0,i)\Lambda_{A}(i,0).\label{eq:W00-1}
\end{equation}
Hence, 
\begin{eqnarray*}
V_{00}\otimes\Lambda_{A}(0,0) & = & V_{00}\otimes\sum_{i\in I}\Lambda_{A}(0,i)\Lambda_{A}(i,0)\\
 & = & \sum_{i\in I}(V_{0i}\otimes\Lambda_{A}(0,i))(V_{i0}\otimes\Lambda_{A}(i,0).
\end{eqnarray*}
Thus, $X\otimes\lambda=\sum_{i\in I}(X_{0i}\otimes\lambda_{0i})(X_{i0}\otimes\lambda_{i0})$
and $Y\otimes\mu=\sum_{i\in I}(Y_{0i}\otimes\mu_{0i})(Y_{i0}\otimes\mu_{i0})$,
where $X_{0i}\otimes\lambda_{0i},Y_{0i}\otimes\mu_{0i}\in V_{0i}\otimes\Lambda_{A}(0,i)$
and $X_{i0}\otimes\lambda_{i0},Y_{i0}\otimes\mu_{i0}\in V_{i0}\otimes\Lambda_{A}(i,0)$.
By \noun{(\ref{eq:W'ij}), }we have $V_{ij}\otimes\Lambda_{A}(i,j)=V'_{ij}\otimes\Lambda_{A}(i,j)\subseteq L$
for all $i\neq j$, so 
\[
X_{0i}\otimes\lambda_{0i},X_{i0}\otimes\lambda_{0i},Y_{0i}\otimes\mu_{0i},Y_{i0}\otimes\mu_{i0}\in L.
\]

We wish to show that $[X\otimes\lambda,Y\otimes\mu]\in L$. Note that
\begin{eqnarray*}
[X\otimes\lambda,Y\otimes\mu] & = & \underset{(i,j)\neq(0,0)}{\sum}[(X_{0i}\otimes\lambda_{0i})(X_{i0}\otimes\lambda_{i0}),(Y_{0j}\otimes\mu_{0j})(Y_{j0}\otimes\mu_{j0})]\\
 & = & \underset{(i,j)\neq(0,0)}{\sum}[X_{0i}X_{i0}\otimes\lambda_{0i}\lambda_{i0},Y_{0j}Y_{j0}\otimes\mu_{0j}\mu_{j0}],
\end{eqnarray*}
so without loose of generality it is enough to show that 
\[
[X_{01}X_{10}\otimes\lambda_{01}\lambda_{10},Y_{01}Y_{10}\otimes\mu_{01}\mu_{10}]\in L.
\]
Put $x_{0}=X_{01}X_{10}\otimes\lambda_{01}\lambda_{10}$ and $y_{0}=Y_{01}Y_{10}\otimes\mu_{01}\mu_{10}$.
Then we need to show that $[x_{0},y_{0}]\in L$. Note that $x_{0},y_{0}\in V_{00}\otimes\Lambda_{A}(0,0)$.
Let 
\[
x_{L}=[X_{01}\otimes\lambda_{01},X_{10}\otimes\lambda_{10}]=X_{01}X_{10}\otimes\lambda_{01}\lambda_{10}-X_{10}X_{01}\otimes\lambda_{10}\lambda_{01}=x_{0}-x_{1},
\]
where $x_{1}=X_{10}X_{01}\otimes\lambda_{10}\lambda_{01}$ and 
\[
y_{L}=[Y_{01}\otimes\mu_{01},Y_{10}\otimes\mu_{10}]=Y_{01}Y_{10}\otimes\mu_{01}\mu_{10}-Y_{10}Y_{01}-\mu_{10}\mu_{01}=y_{0}-y_{1},
\]
where $y_{1}=Y_{10}Y_{01}\otimes\mu_{10}\mu_{01}$. Then $x_{L},y_{L}\in[L,L]=L$,
and $x_{1},y_{1}\in V_{11}\otimes\Lambda_{A}(1,1)$. Since $x_{0}=x_{L}+x_{1}$
and $y_{0}=y_{L}+y_{1}$, we have 
\begin{eqnarray}
[x_{0},y_{0}] & = & [x_{L}+x_{1},y_{L}+y_{1}]\nonumber \\
 & = & [x_{L},y_{L}]+[x_{L},y_{1}]+[x_{1},y_{L}]+[x_{1},y_{1}].\label{eq:=00005Bx0,y0=00005D-1}
\end{eqnarray}

To show that $[x_{0},y_{0}]\in L$, it suffices to show that each
term on the right side of (\ref{eq:=00005Bx0,y0=00005D-1}) is in
$L$. We have $[x_{L},y_{L}]\in[L,L]=L$. As $i\in I$, by Lemma \ref{lem:Vii * ^(i,i) },
$[x_{1},y_{1}]\in L$. Moreover, by Lemma \ref{lem:=00005BL,Vii * ^A(i,i)=00005D sub L},
$[L,V_{ii}\otimes\Lambda_{A}(i,i)]\subseteq L$, so $[x_{L},y_{1}],[x_{1},y_{L}]\in L$.
Therefore, $[x_{0},y_{0}]\in L$. Thus, $[X\otimes\lambda,Y\otimes\mu]\in L$
for all $i$, $j$, $s$ and $t$, as required. 
\end{proof}
By combining Lemma \ref{lem:Vii * ^(i,i) } and Lemma \ref{lem:V00 * ^(0,0) },
we obtain the following. 
\begin{prop}
\label{prop:Vii * ^(i,i) for all i in I hat} $[V_{ii}\otimes\Lambda_{A}(i,i),V_{ii}\otimes\Lambda_{A}(i,i)]\subseteq L$
for all $i\in\hat{I}$. 
\end{prop}

\begin{thm}
\label{thm:L=00003D=00005BA,A=00005D is perfect} $A^{(1)}$ is perfect
and $A=A^{(1)}A^{(1)}+A^{(1)}$. 
\end{thm}

\begin{proof}
We identify $A$ with $\bigoplus_{i,j\in\hat{I}}V_{ij}\otimes\Lambda_{A}(i,j)$.
By Lemma \ref{prop:Vij.^(i,j )  is an algebra}, $S=\bigoplus_{i\in I}S_{i}=\bigoplus_{i\in I}V_{ii}\otimes\mathbf{1}_{i}$
is a Levi subalgebra of $A$. As above we fix a matrix realization
$M_{n_{i}}(\bbF)$ for each simple component $S_{i}$ of $S$. Let
$L$ be a Lie subalgebra of $A$ generated by all $V'_{i,j}\otimes\Lambda_{A}(i,j)$
for $(i,j)\neq(0,0)$, where $V'_{ij}$ is defined in (\ref{eq:W'ij}).
Then by Proposition \ref{prop:L is perfect}, $L$ is perfect and
$Q=[S,S]=\bigoplus_{i\in I}V'_{ii}\otimes\mathbf{1}_{i}\subseteq L$.
We wish to show that $L=A^{(1)}$. Since $L$ is a perfect Lie subalgebra
of $A^{(-)}$, we have $L=L^{(1)}\subseteq A^{(1)}$. It remains to
show that $A^{(1)}\subseteq L$. Let $X\otimes\lambda\in V_{ij}\otimes\Lambda_{A}(i,j)$
and $Y\otimes\mu\in V_{st}\otimes\Lambda_{A}(s,t)$, where $i,j,s,t\in\hat{I}$.
We wish to show that 
\begin{equation}
[X\otimes\lambda,Y\otimes\mu]=XY\otimes\lambda\mu-YX\otimes\mu\lambda\in L.\label{eq:=00005BX,Y=00005D=00003DXY-YX}
\end{equation}

If $i\neq t$ and $j\neq s$, then obviously we have $[X\otimes\lambda,Y\otimes\mu]=0\in L$.
Suppose that $j=s$. First, consider the case when $i\neq t$. Then
$YX\otimes\mu\lambda=0$. By using the formula (\ref{eq:W'ij}) we
obtain $XY\otimes\lambda\mu\in V_{it}\otimes\Lambda_{A}(i,t)=V'_{it}\otimes\Lambda_{A}(i,t)\subseteq L$.
Therefore, $[X\otimes\lambda,Y\otimes\mu]=XY\otimes\lambda\mu\in L$. 

Now, suppose that $i=t$. If $i\neq j$, then by formula (\ref{eq:W'ij}),
$V_{ij}\otimes\Lambda_{A}(i,j)=V'_{ij}\otimes\Lambda_{A}(i,j)\subseteq L$
for all $i\neq j$. Hence, 
\[
[X\otimes\lambda,Y\otimes\mu]\in[V'_{ij}\otimes\Lambda_{A}(i,j),V'_{ji}\otimes\Lambda_{A}(j,i)]\subseteq[L,L]=L.
\]
Suppose that $i=j$. Then by Proposition \ref{prop:Vii * ^(i,i) for all i in I hat},
$[V_{ii}\otimes\Lambda_{A}(i,i),V_{ii}\otimes\Lambda_{A}(i,i)]\in L$
for all $i\in\hat{I}$. Hence, $[X\otimes\lambda,Y\otimes\mu]\in L$
for all $i$, $j$, $s$ and $t$. Therefore, $A^{(1)}=L$ is perfect,
as required. 

It remains to show that $A=LL+L$. Note that $V'_{ij}\otimes\Lambda_{A}(i,j)\subseteq L\subseteq A$
for all $(i,j)\neq(0,0)$. Suppose that $j\neq0$. Then for all $\lambda\in\Lambda_{A}(i,j)$
we have $(e_{11}\otimes\lambda)(e_{12}\otimes\mathbf{1}_{i})\in(V'_{ij}\otimes\lambda)(V'_{jj}\otimes\mathbf{1}_{j})\subseteq LQ\subseteq LL$.
Therefore, $V_{ij}\otimes\Lambda_{A}(i,j)\subseteq LL+L$ for all
$(i,j)\neq(0,0)$. Now, assume that $j=0$. Then either $i\neq0$
or $i=0$. The case when $i\neq0$ is similar to above. Suppose that
$i=0$. Then by using Proposition \ref{prop:(i) =00003D (ii) =00003D (iii)}
and the formula (\ref{eq:W'ij}) we obtain 
\begin{eqnarray*}
V_{00}\otimes\Lambda_{A}(0,0) & = & V_{00}\otimes\sum_{i\in I}\Lambda_{A}(0,i)\Lambda_{A}(i,0)\\
 & = & \sum_{i\in I}(V_{0i}\otimes\Lambda_{A}(0,i))(V_{i0}\otimes\Lambda_{A}(i,0)\\
 & = & \sum_{i\in I}(V'_{0i}\otimes\Lambda_{A}(0,i))(V'_{i0}\otimes\Lambda_{A}(i,0)\subseteq LL.
\end{eqnarray*}
Therefore, $A=LL+L$. 
\end{proof}
\begin{proof}[Proof of Theorem \ref{thm:main}]
 This follows from Proposition \ref{prop:(i) =00003D (ii) =00003D (iii)}
and Theorem \ref{thm:L=00003D=00005BA,A=00005D is perfect}. 
\end{proof}
\begin{proof}[Proof of Theorem \ref{thm:fd}]
 Let $S$ be a Levi subsalgebra of $A$. Then by Proposition \ref{levi},
$A$ is $S$-modgenerated. Since $S\cong A/\rad(A)$, the algebra
$S$ is $k$-perfect whenever $A$ is $k$-perfect. Therefore the
result follows from Theorem \ref{thm:L=00003D=00005BA,A=00005D is perfect}.

\end{proof}


\begin{thebibliography}{1}
\bibitem{BaBZ} Y. Bahturin, A. Baranov, and A. Zalesski. Simple Lie
subalgebras of locally finite associative algebras. J. Algebra, 281(1),
(2004), 225-{}-246. 

\bibitem{Bav=000026Zal} A. Baranov and A. Zalesskii. Plain representations
of Lie algebras. Journal of the London Mathematical Society 63 (3),
(2001), 571-591.

\bibitem{Bav=000026Zal-1} A. Baranov and A. Zalesskii. Quasiclassical
Lie Algebras, Journal of Algebra 243, (2001), pp. 264-293.

\bibitem{berman1992lie} S.\textasciitilde Berman and R.\textasciitilde Moody.
Lie algebras graded by finite root systems and the intersection matrix
algebras of Slodowy. Inventiones mathematicae, 108(1) (1992), 323-{}-347. 

\bibitem{Herstein} I. Herstein, Lie and Jordan structures in simple
associative rings, Bulletin of the American Mathematical Society 67.6
(1961), pp. 517-531.
\end{thebibliography}
\end{document}